\documentclass[psamsfonts,reqno,11pt]{amsart}
\usepackage{graphicx}
\usepackage{graphics}
\usepackage{latexsym}
\usepackage{amsmath}
\usepackage{amssymb}

\newtheorem{thm}{Theorem}[section]
\newtheorem{lem}[thm]{Lemma}

\newtheorem{prop}[thm]{Proposition}
\newtheorem{cor}[thm]{Corollary}

\newtheorem{rem}[thm]{Remark}

\newtheorem{rem-eg}[thm]{Remark and Example}

\newtheorem{prob}[thm]{Problem}

\newenvironment{prf}{{\noindent \textbf{Proof:} }}{\hfill $\Box$\medskip}
\def\sideremark#1{\ifvmode\leavevmode\fi\vadjust{\vbox
to0pt{\vss \hbox to 0pt{\hskip\hsize\hskip1em
\vbox{\hsize2cm\tiny\raggedright\pretolerance10000
\noindent#1\hfill}\hss}\vbox to8pt{\vfil}\vss}}}

\begin{document}

\title{Phase-isometries on real normed spaces }

\author[Dongni Tan]{Dongni Tan}
\author[Xujian Huang]{Xujian Huang $^{*}$}
\address[Dongni Tan]{A Department of Mathematics, Tianjin University of Technology, Tianjin 300384, P.R. China}
\email{tandongni0608@sina.cn}
\address[Xujian Huang]{A Department of Mathematics, Tianjin University of Technology, Tianjin 300384, P.R. China}
\email{huangxujian86@sina.cn}

\keywords{phase-isometry; linear isometry; Mazur-Ulam theorem;  Wigner's theorem}

\subjclass[2010]{primary 46B03, secondly 46B04}

\thanks{
$*$Corresponding author.\\
\indent The authors are supported by the Natural Science Foundation of China (Grant Nos. 11371201.11201337, 11201338, 11301384).}

\maketitle

\begin{abstract}
We say that a mapping $f: X \rightarrow Y$ between two real normed spaces is a phase-isometry if it satisfies
the functional equation
\begin{eqnarray*} \{\|f(x)+f(y)\|, \|f(x)-f(y)\|\}=\{\|x+y\|, \|x-y\|\} \quad  (x,y\in X).\end{eqnarray*}
A generalized Mazur-Ulam question is whether every surjective phase-isometry is a multiplication of a linear isometry and a map with range $\{-1, 1\}$. This assertion is also an extension of a fundamental statement in the mathematical description of quantum mechanics, Wigner's theorem to real normed spaces. In this paper, we show that for every space $Y$ the problem is solved in positive way if $X$ is a smooth normed space, an $\mathcal{L}^{\infty}(\Gamma)$-type space or an $\ell^1(\Gamma)$-space with $\Gamma$ being an index set.
\end{abstract}
\section{Introduction}

\medskip

The study of isometries between normed spaces dates back to 1930s. The classical Mazur-Ulam theorem (\cite{MU}, 1932) states that an isometry
$f$ of a real normed space onto another normed space with $f(0)=0$ is linear. In general, the Mazur-Ulam Theorem fails without surjectivity assumption. For example, consider the mapping $f:\mathbb{R}\rightarrow \ell^2_{\infty}$ defined by $f(t)=(t, sint)$, which shows that an into isometry $f$ with $f(0)= 0$ is not necessarily linear. In 1967, Figiel \cite{Fi} proved the following non-surjective substitute for the Mazur-Ulam theorem: for every isometry $f$ of a real normed space $X$ into another normed space $Y$ with $f(0)=0$, there is a linear operator $T$ of norm one from $\overline{\mbox{span}}f(X)$ onto $X$ such that $T\circ f$ is the identity on $X$. Recently, Cheng et al. \cite{CDZ} gave a quantitative extension of Figiel's theorem: for every $\varepsilon$-isometry ($\varepsilon\geq 0$) of a real normed space $X$ into another normed space $Y$ with $f(0)=0$ and every $x^*\in X^*$, there exists $\varphi\in Y^*$ with $\|\varphi\|=\|x^*\|$ such that
\begin{align*}
|\varphi(f(x))- x^*(x)|\leq 4\varepsilon\|x^*\|\quad (x\in X).
\end{align*}
Figiel's theorem and its generalizations play an important role in the study of isometric embedding and $\varepsilon$-isometric approximative problems (see, for instance \cite{CDZ, CDD, DL, GK, OS, Q, SV}). These results show some deep relationship between isometries, $\varepsilon$-isometries  and linear isometries.

The famous Wigner's theorem is related to linear isometries, which plays a fundamental role in quantum mechanics and has deep connections with the mathematical theory of projective spaces. There are several equivalent formulations of Wigner's theorem, see \cite{B, BMS, Ge, G, M, SA, T} to list just some of them. One of them (see \cite{G} or \cite{SA}) characterizes mappings on a Hilbert space preserving the absolute value of the inner product of any pair of vectors. That is, assuming that $H$ and $K$ are complex Hilbert spaces and $f: H \rightarrow K $ is a mapping satisfying the functional equation
\begin{eqnarray}\label{abs}
|<f(x), f(y)>|=|<x, y>|  \quad (x,y\in H).
\end{eqnarray} Then there exists a phase function $\varepsilon: H \rightarrow \mathbb{C}$ with $|\varepsilon(x)|=1$ such that $\varepsilon f$ is a linear or conjugate linear isometry. In the real setting, Wigner's theorem says that every mapping $f$ satisfying the equation (\ref{abs}) is phase equivalent to a linear isometry (i.e., there exists a phase function $\varepsilon: H \rightarrow \{-1, 1\}$ such that $\varepsilon f$ is a linear isometry).  It is worth mentioning that L.Monl$\acute{a}$r \cite{M} described the form of all bijective mappings on the set of all rank-one idempotents of a Banach space which preserve zero products in both directions. This generalizes Uhlhorn's version of Wigner's theorem to indefinite inner product spaces (see the introduction of \cite{BMS}). From a mathematical point of view, it has raised a need to study Wigner's theorem in  more general setting of Banach spaces. It is worth studying Wigner's theorem from various points of view and a new general version of Wigner's theorem may certainly improve our understanding of it.

Let $X$ and $Y$ be real normed spaces. A mapping $f: X\rightarrow Y$ is called a \emph{phase-isometry} if it satisfies the functional equation \begin{eqnarray} \label{phase}
\{\|f(x)+f(y)\|, \|f(x)-f(y)\|\}=\{\|x+y\|, \|x-y\|\} \quad ( x,y\in X).
\end{eqnarray}
Let us say that a mapping $f: X \rightarrow Y$ is \emph{phase equivalent to a linear isometry} if there exists a phase function $\varepsilon: X \rightarrow \{-1, 1\}$ such that $\varepsilon f$ is a linear isometry. We can easily see that every mapping that is phase equivalent to a linear isometry is definitely a phase-isometry. It is interesting in seeing whether the converse also holds for real normed spaces. Motivated by the Mazur-Ulam Theorem, the following natural problem arises.
\begin{prob}\label{wmu}
Let $X$ and $Y$ be real normed spaces, and let $f: X \rightarrow Y$ be a surjective phase-isometry. Is it true that $f$ is phase equivalent to a linear isometry?
\end{prob}

By Wigner's theorem, this problem is solved in positive way without onto assumption for real Hilbert spaces. Therefore, we can regard Problem \ref{wmu} as a generalized Wigner's theorem for real normed spaces. Huang and Tan studied surjective phase-isometries between the same typical classical normed spaces such as $\ell^p(\Gamma)$ spaces with $(0<p<\infty)$ \cite{HT} and $\mathcal{L}^{\infty}(\Gamma)$-type spaces \cite{JT}, and they got positive answers for Problem \ref{wmu}. It should be noted that $l^p(\Gamma)$ spaces for all $0<p<1$  are not normed spaces. The proofs for such special normed spaces were obtained by using the specific form of norms and a lot of special techniques. The Problem \ref{wmu} for general real normed spaces is easy to understand, but it is not at all evident. Since continuity is implied by isometries, the proof of Mazur-Ulam theorem consists of showing that the surjective mapping preserves the midpoint of every segment. However, a phase-isometry is not necessarily continuous. The main feature of our approach is that instead of using the original proof of Mazur-Ulam Theorem, we shall rely on a modification of the quantitative extension of Figiel's theorem \cite{CDZ}. We will also need the fundamental theorem of projective geometry.

The aim of this paper is to contribute to the study of giving a general Wigner-type theorem on real normed spaces. We will apply the fundamental theorem of projective geometry to present a sufficient condition such that the Problem \ref{wmu} can be solved in the positive for general normed spaces. We give a new quantitative extension of Figiel's theorem in corresponding to phase-isometries, and show that for every phase-isometry $f$ of a real normed space $X$ into another normed space $Y$ and every $w^*$-exposed point $x^*$ of $B_{X^*}$, there exists a linear functional $\varphi \in Y^*$ of norm one such that $x^*(x)=\pm \varphi(f(x))$ for all $x\in X$. Making use of this result, we prove that if $X$ is smooth then every surjective phase-isometry $f: X \rightarrow Y$ is phase equivalent to a linear isometry. As a particular case, Problem \ref{wmu} is solved positively provided $X$ is an $\mathcal{L}^p(\mu)$ space with $1<p<\infty$ for an arbitrary measure $\mu$. With the properties of Birkhoff orthogonality, we still present another sufficient condition such that Problem \ref{wmu} can be solved in the positive for some special normed spaces. This will help us to deal with some sequence spaces. As a result, we prove that if $X$ is an $\mathcal{L}^{\infty}(\Gamma)$-type space or an $\ell^1(\Gamma)$-space for some index set $\Gamma$ then every surjective phase-isometry $f: X \rightarrow Y$ is phase equivalent to a linear isometry. These results extend the famous Wigner's theorem to real normed spaces, and can also be considered as a generalization of Mazur-Ulam type results. Although we feel this result is interesting in its own right, we hope that it will serve as a stepping stone to showing that every phase-isometry from a real normed space onto another normed space is phase equivalent to a linear isometry.

\section{Main lemma and Phase-isometries for smooth normed spaces}

In this note, the letters $X, Y$ are used to denote real normed spaces, $X^*, Y^*$ are their dual spaces. For a real normed space $X$ we denote by $S_X$ and $B_X$ the unit sphere and the closed unit ball of $X$, respectively.

We start this section with an interesting result showing that every surjective phase-isometry is an injective odd norm-preserving map.

\begin{lem}\label{inj}
Let $X$ and $Y$ be real normed spaces, and let $f: X \rightarrow Y$ be a phase-isometry. Then $f$ is a norm-preserving map and $f(-x)=\pm f(x)$ for every $x\in X$. Moreover, if $f$ is surjective then it is injective and $f(-x)=-f(x)$ for all $x\in X$.
\end{lem}
\begin{prf}
With the substitution $y=x$ in the equation (\ref{phase}), it follows that $f$ is norm preserving. Putting $y=-x$ in the equation (\ref{phase}), this yields \begin{eqnarray*}\{\|f(x)+f(-x)\|, \|f(x)-f(-x)\|\}=\{2\|x\|, 0\},\end{eqnarray*}
which implies that $f(-x)=\pm f(x)$. To prove the second conclusion, suppose that $f$ is surjective and $f(x)=f(y)$ for some $x, y \in X$. Using the norm preserving property, we have $f(x)=0$ if and only if $x=0$. Assume that $f(x)=f(y)\neq 0$. Choose $z\in X$ such that $f(z)=-f(x)$. Using the equation (\ref{phase}) for $x, y, z$, we obtain that
\begin{eqnarray*}&\{\|x+y\|, \|x-y\|\}=\{\|f(x)+f(y)\|, \|f(x)-f(y)\|\}=\{2\|x\|, 0\}\\
&\{\|x+z\|, \|x-z\|\}=\{\|f(x)+f(z)\|, \|f(x)-f(z)\|\}=\{2\|x\|, 0\}.\end{eqnarray*} This yields $y, z\in \{x, -x\}$. If $z=x$, then $f(x)=-f(x)=0$, which is a contradiction. So we obtain $z=-x$, and we must have $y=x$. For otherwise, we get $y=-x=z$ and \begin{eqnarray*}f(x)=f(y)=f(z)=-f(x).\end{eqnarray*} This leads to the contradiction that $f(x)\neq 0$.
\end{prf}

By the above result, a natural question may be raised whether the surjective phase-isometry is an isometry. An easy example where $f(x_0)=-x_0$, $f(-x_0)=x_0$ and $f(x)=x$ if $x\neq \pm x_0$ for some nonzero vector $x_0\in X$  indicates that it is not the case.

We will give an affirmative answer to the Problem \ref{wmu} when the domain $X$ is one dimensional.

\begin{prop}\label{real}
Let $X$ and $Y$ be two real normed spaces with $X$ being one dimensional, and let $f: X \rightarrow Y$ be a surjective phase-isometry. Then $f$ is phase-equivalent to a linear isometry.
\end{prop}

\begin{prf}
Let $x$ be a norm-one vector in $X$. Note that $f$ is surjective norm-preserving map and that $X$ is one dimensional. For every $t\in \mathbb{R}$, if $f(y)=tf(x)$ for some $y\in X$ then $y=\pm tx$. Define a mapping $\varepsilon: X \rightarrow \{-1, 1\}$ as following: $\varepsilon(0)=1$, $\varepsilon(tx)$ is equal to $1$ or $-1$ according as $f(tx)=tf(x)$ or $f(tx)=-tf(x)$ for every $0\neq t\in \mathbb{R}$. Then the mapping $g:=\varepsilon f$ is a linear isometry as desired.
\end{prf}

We shall apply the fundamental theorem of projective geometry. In \cite{F}, a version of this theorem is stated for vector spaces over division rings. For the convenience of the readers, we write here a slight modification of the fundamental theorem for real vector spaces, which comes from \cite[Theorem 3]{Ge1}. Let $X$ be a real vector space. For every $M\subset X$, $[M]$ will denote the subspace generated by the set $M$. In particular, for every $0\neq x\in X$, $[x]$ will denote the line $\mathbb{R}\cdot x$.

\begin{thm}\label{pro} Let $X$ and $Y$ be two real vector spaces with at least three dimensions, and let $f: X\rightarrow Y$ be a mapping  which satisfies the following conditions:\\
\emph{(a)} The range of $f$ is not contained in a two-dimensional subspace of $Y$;\\
\emph{(b)} $z\in[x, y]$ implies that $f(z)\in[f(x), f(y)].$\\
Then there exists an injective linear map $A:X\rightarrow Y$ such that
\begin{equation*}
  [f(x)]=[A(x)], \quad \forall x\in X.
\end{equation*}
\end{thm}

We shall present a sufficient condition such that the Problem \ref{wmu} can be solved in positive way for general normed spaces. The proof of this result relies on the fundamental theorem of projective geometry.

\begin{prop}\label{isometry}
Let $X$ and $Y$ be real normed spaces, and let $f: X \rightarrow Y$ be a surjective phase-isometry satisfying the following conditions:\\
\emph{(a)} $f(t x)=\pm tf(x)$ for every $x\in X$ and $t\in \mathbb{R}$;\\
\emph{(b)} For all two linearly independent vectors $x$ and $y$ in $X$, there exist two real numbers $\alpha$ and $\beta$ with $|\alpha|=|\beta|=1$
such that $f(x+y)=\alpha f(x)+\beta f(y)$.\\
Then $f$ is phase-equivalent to a linear isometry.
\end{prop}

\begin{prf}
We first prove that $f$ is phase equivalent to a homogeneous mapping from condition (a). Indeed, by the axiom of choice, there is a set $L$ such that for every $0\neq x\in X$ there exists exactly one element $y\in L$ such that $x=s y$ for some $s \in \mathbb{R}$. Define $f_0: X \rightarrow Y$ by
\begin{eqnarray*} f_0(x)=f_0(s y)=s f(y), \quad \forall x=s y\in X.\end{eqnarray*} Then $f_0$ is well defined, homogeneous and phase equivalent to $f$ by condition (a). Now we may assume that $f$ is homogeneous.

If $X$ is one dimensional, then the conclusion is immediate from Corollary \ref{real}.

Suppose for the moment that $X$ is at least three dimensions. We will prove that $\dim Y \geq 3$. Suppose that, on the contrary, we have $\dim Y \leq 2$. Take an arbitrary $x\in S_X$, and set $M:=\{y\in S_X: \|x-y\|=\|x+y\|\}$ and $L:=\{w\in S_Y: \|w+f(x)\|=\|w-f(x)\|\}$. Since $\dim X\geq 3$, the set $M$ contains infinite elements (see the proof of \cite[Lemma 2.3]{Ry} for details). From equation (\ref{phase}), we have $f(y)\in L$ for every $y\in M$. As a consequence, the set $L$ contains infinite elements, because $f$ is injective by Lemma \ref{inj}. This is in contradiction with the fact that
the set $L$ contains at most two elements (see \cite[Theorem 4.35]{AMW}). This contradiction implies that $\dim Y \geq 3$. Therefore we would be able to apply Theorem \ref{pro} and find an injective linear map $A:X\rightarrow Y$ such that $[f(x)]=[A(x)]$ for every $x\in X$. Consequently, there is a function $\lambda: X\rightarrow \mathbb{R}$ such that $f(x)=\lambda(x) A(x)$ for every $x\in X$. Since $f$ is homogeneous, we have $\lambda(tx)=\lambda(x)$ for every $x\in X$ and $0\neq t\in \mathbb{R}$. Moreover, suppose that $x, y\in X$ are two linearly independent vectors. We write $f(x+y)=\alpha f(x)+\beta f(y)$ for two real numbers $\alpha$ and $\beta$ with $|\alpha|=|\beta|=1$.
We immediately obtain
\begin{eqnarray*}
\alpha\lambda(x)A(x)+\beta \lambda(y)A(y)=f(x+y)=\lambda(x+y)A(x)+\lambda(x+y)A(y).
\end{eqnarray*}
It follows that $\lambda(x+y)= \alpha\lambda(x)=\beta \lambda(y)$.
As a consequence, the $|\lambda(x)|$ has only one value which is denoted by $\lambda$. Hence we can define a desired phase function $\varepsilon :X\rightarrow \{-1,1\}$ such that $f=\varepsilon \lambda A$.
So $f$ is phase equivalent to the linear isometry $\lambda A$.

Now suppose that $X$ is two dimensional. It follows from \cite[Theorem 4.7 and Corollary 4.2]{AMW} that there exist a basis $\{x, y\}\subset S_X$ of $X$ and two linear functionals $x^*, y^*\in X^*$ of normed one such that
\begin{eqnarray*}x^*(x)=y^*(y)=1\ \mbox{and} \ x^*(y)=y^*(x)=0.\end{eqnarray*}
By condition (b), we can write
\begin{eqnarray}\label{hom1}
f(ax+y)=\alpha(a)f(ax)+\beta(a)f(y),\ |\alpha(a)|=|\beta(a)|=1
\end{eqnarray}
for every $a\in \mathbb{R}$.
Since $f$ is homogeneous, for every $a, b\,(\neq 0) \in \mathbb{R}$,
\begin{eqnarray}\label{hom2}
f(ax+by)=b f(\frac{ax}{b}+y)=\alpha(\frac{a}{b}) f(ax)+ \beta(\frac{a}{b}) f(by) ,\ |\alpha(\frac{a}{b})|=|\beta(\frac{a}{b})|=1.\end{eqnarray}
Define a mapping $g: X \rightarrow Y$ as follows:
\begin{eqnarray*}
g(0)=0, \ g(ax)=\alpha(a)\beta(a)f(ax),\ \ g(ax+by)=\alpha(\frac{a}{b})\beta(\frac{a}{b}) f(ax)+f(by)
\end{eqnarray*}
for all $a, b \,(\neq0)\in\mathbb{R}$. Then the relations (\ref{hom1}) and (\ref{hom2}) show that $g$ is a phase-isometry, and it is phase equivalent to $f$. We will prove that \begin{equation}\label{equ:6}
 \alpha(a)\beta(a)=\alpha(1)\beta(1)
\end{equation}
for all $0\neq a\in\mathbb{R}$.
The equation (\ref{equ:6}) and the definition of $g$ show that \begin{eqnarray*}g(ax+by)=ag(x)+bg(y) \quad ( a,b\in \mathbb{R}).\end{eqnarray*}
This means that $g$ is a linear isometry from $X$ onto $Y$. Thus, to see our conclusion, we only need to prove the equation (\ref{equ:6}). For this, we need the following
three equations:
\begin{eqnarray}
\label{a1} &&\{\|g(ax+y)\pm g(ax-y)\|\}=\{|2a|, 2\}, \\
\label{a2}&&\{\|g(ax+y)\pm g(x+y)\|\}=\{\|(a+1)x+2y\|, |a-1|\},\\
\label{a3}&&\{\|g(ax+ay)\pm g(ax+y)\|\}=\{\|2ax+(a+1)y\|, |a-1|\}
\end{eqnarray}
for every $a\in\mathbb{R}$.
Indeed, from the equation (\ref{a1}), we have  \begin{eqnarray*}|a|\cdot |\alpha(a)\beta(a)+\alpha(-a)\beta(-a)|=\|g(ax+y)+g(ax-y)\|\in \{|2a|, 2\},\end{eqnarray*} which yields
\begin{eqnarray*}\alpha(a)\beta(a)=\alpha(-a)\beta(-a) \quad (0\neq a \in\mathbb{R}).\end{eqnarray*}
If $0<a<1$, from the equation (\ref{a2}),
\begin{eqnarray*}|a\alpha(a)\beta(a)-\alpha(1)\beta(1)|=\|g(ax+y)-g(x+y)\|\in \{\|(a+1)x+2y\|, 1-a\}.\end{eqnarray*}
The inequality $\|(a+1)x+2y\|\geq y^*((a+1)x+2y)=2>a+1$ proves that
\begin{eqnarray*}\alpha(a)\beta(a)=\alpha(1)\beta(1) \quad \mbox{for all} \ 0<a<1.\end{eqnarray*}
If $a>1$, from the equation (\ref{a3}),
\begin{eqnarray*}
\|(\alpha(1)\beta(1)+\alpha(a)\beta(a))f(ax)+(a+1)f(y)\|\in \{\|2ax+(a+1)y\|, a-1\}.
\end{eqnarray*}
The inequality $\|2ax+(a+1)y\|\geq x^*(2ax+(a+1)y)=2a>a+1$ implies that
\begin{eqnarray*}\alpha(a)\beta(a)=\alpha(1)\beta(1) \quad \mbox{for all} \ a>1.\end{eqnarray*}
The proof is complete.
\end{prf}

\begin{rem} \label{weak}
Let $X$ and $Y$ be real normed spaces, and let $f: X \rightarrow Y$ be a surjective phase-isometry. It follows from Lemma \ref{inj} that $f(-x)=-f(x)$ for every $x\in X$. Then the conditions (a) and (b) of Proposition \ref{isometry} are equivalent to the following weak conditions:\\
\emph{(a*)} $f(t x)=\pm tf(x)$ for every $x\in S_X$ and $t>0$;\\
\emph{(b*)} For all two linearly independent vectors $x$ and $y$ in $B_X$, there exist two real numbers $\alpha$ and $\beta$ with $|\alpha|=|\beta|=1$ such that $f(x+y)=\alpha f(x)+\beta f(y)$.
\end{rem}

To show the main lemma of this paper, we recall some notations and results. For every $0\neq x\in X$, denote by $D(x)$ the set of supporting functionals of $x$; that is \begin{eqnarray*}D(x)=\{x^*\in S_{X^*}: x^*(x)=\|x\|\}.\end{eqnarray*}
A point $x\in X$ is called \emph{smooth} if there exists only one supporting functional at $x$. A normed space $X$ is said to be \emph{smooth} provided that every nonzero element $x\in X$ is a smooth point. Examples of smooth normed spaces are $\mathcal{L}^p$-spaces for $1<p<\infty$. However $\mathcal{L}^1$-spaces and $\mathcal{L}^{\infty}$-spaces are not smooth.

For every $u\in S_X$ and $x\in X$, denote by $M_u(x)$ the directional derivative of the function $x\rightarrow \|x\|$ at the point $u$ in the direction $x$: \begin{eqnarray*}
M_u(x)=\lim_{t \rightarrow 0^+}\frac{\|u+tx\|-\|u\|}{t}.
\end{eqnarray*}
By the convexity of the function $x\rightarrow \|x\|$ is convex, the directional derivative exists. In general $M_u: X\rightarrow \mathbb{R}$ is not linear, but it is sub-linear (see \cite[Lemma 1.2]{P}). It is well known that (see \cite[Proposition 2.24]{P})
\begin{eqnarray*}
M_u(x)=\max\{x^*(x): x^*\in D(u)\}.
\end{eqnarray*}
In particular, if $u$ is a smooth point then $M_u$ is the unique supporting functional at $u$.

A point $x^*$ in a $w^*$-closed convex set $C\subset X^*$ is said to be a $w^*$-exposed point of $C$ provided that there exists $x\in X$ such that $x^*(x)>y^*(x)$ for all $y^*\in C$ with $y^*\neq x^*$. It is easy to observe that $x^*\in X^*$ is a $w^*$-exposed point of the dual unit ball $B_{X^*}$ if and only if $x^*$ is the only one supporting functional for some smooth point $x\in S_X$.

We will give a modification of Figiel's theorem as corresponding to equation (\ref{phase}). This result is of its own interest but could also be used to give some affirmative answers to Problem \ref{wmu}.

\begin{lem}(Main lemma) \label{diff}
Let $X$ and $Y$ be real normed spaces, and let $f: X \rightarrow Y$ be a phase-isometry (not necessarily surjective). Then for every $w^*$-exposed point $x^*$ of $B_{X^*}$, there exists a linear functional $\varphi \in Y^*$ of norm one such that $x^*(x)=\pm \varphi(f(x))$ for all $x\in X$.
\end{lem}

\begin{prf}
The proof is based upon an idea of \cite[lemma 2.4]{CDZ} for a special case $\varepsilon=0$.

We first prove that if $X=\mathbb{R}$ then there is a linear functional $\varphi \in Y^*$ of norm one such that $\varphi(f(t))=\pm t$ for all $t\in \mathbb{R}$. For every positive integer $n$, using the norm preserving property from Lemma \ref{inj}, we have $\|f(n)\|=n$. The Hahn-Banach theorem guarantees the existence of a linear functional $\varphi_n\in S_{Y^*}$ such that $\|\varphi_n\|=1$ and $\varphi_n(f(n))= n$. For every $t\in [-n, n]$, \begin{eqnarray*}
2n&=&\varphi_n (f(n)-f(t))+\varphi_n(f(n)+f(t))\\&\leq &\|f(n)-f(t)\|+\|f(n)+f(t)\|\\ &=&(n-t)+(n+t)=2n,
\end{eqnarray*}
or alternatively \begin{eqnarray*}&&\{\varphi_n (f(n)-f(t)), \varphi_n(f(n)+f(t))\}\\ &=&\{\|f(n)-f(t)\|, \|f(n)+f(t)\|\}\\ &=&\{n-t, n+t\}.\end{eqnarray*}
Then $\varphi_n(f(t))=\pm t$ for all $t\in [-n, n]$. It follows from Alaoglu's theorem that the sequence $\varphi_n$ has a cluster point $\varphi$ in the $w^*$ topology. This entails that $\|\varphi\|\leq 1$ and $\varphi(f(t))=\pm t$ for every $t\in \mathbb{R}$. Clearly, $\|\varphi\|=1$ and $\varphi$ is the desired mapping.

Now suppose that $\dim(X)>1$ and $u\in S_X$ is a smooth point such that $x^*(u)=1$. Let $g: \mathbb{R} \rightarrow Y$ be defined by for every $t\in\mathbb{R}$, $g(t)=f(tu)$. Then $g$ satisfies the functional equation (\ref{phase}). By the above, there exists $\varphi \in Y^*$ with $\|\varphi\|=1$ such that \begin{eqnarray*} \varphi(f(tu))=\varphi (g(t))=\pm t.\end{eqnarray*}
Since $u$ is a smooth point, it follows that $x^*$ is the only one supporting functional at $u$. Therefore, for every $x\in X$,
\begin{eqnarray*}
x^*(x)=\lim_{t \rightarrow 0^+}\frac{\|u+tx\|-\|u\|}{t}=\lim_{t \rightarrow +\infty}(\|tu+x\|-t).
\end{eqnarray*}
From the equation (\ref{phase}), we get \begin{eqnarray*}\{\|f(tu)+f(x)\|, \|f(tu)-x\|\}=\{\|tu+x\|, \|tu-x\|\}\end{eqnarray*}
for all $t>0$ and $x\in X$. For a fixed nonzero vector $x\in X$, the set $(0, +\infty)$ will be divided into four parts:
\begin{eqnarray*}
&&A_1:=\{t>0:\|f(tu)\pm f(x)\|=\|tu \mp x\|, \ \varphi(f(tu))=t\};\\
&&A_2:=\{t>0:\|f(tu)\pm f(x)\|=\|tu\pm x\|, \ \varphi(f(tu))=t\};\\
&&A_3:=\{t>0:\|f(tu)\pm f(x)\|=\|tu\pm x\|,  \ \varphi(f(tu))=-t\};\\
&&A_4:=\{t>0:\|f(tu)\pm f(x)\|=\|tu\mp x\|,  \ \varphi(f(tu))=-t\}.
\end{eqnarray*}
Obviously, at least one of the sets $\{A_i: i=1,2,3,4\}$ is unbounded. We shall prove that if $A_i$ is unbounded
then $$x^*(x)=(-1)^i \varphi(f(x))$$ for all $i=1,2,3,4$. Without loss of generality we can assume that $A_1$ is unbounded. Then for every $t\in A_1$,
\begin{eqnarray*}
&&\|tu+x\|-t=\|f(tu)-f(x)\|-t\geq \varphi (f(tu)-f(x))-t=-\varphi(f(x)),\\
&&\|tu-x\|-t=\|f(tu)+f(x)\|-t\geq \varphi (f(tu)+f(x))-t=\varphi(f(x)).
\end{eqnarray*}
Let $t\in A_1$ and $t \rightarrow +\infty$ in the two inequalities above, we have $$x^*(x)=-\varphi(f(x)).$$
This completes the proof.
\end{prf}

Specializing Lemma \ref{diff} to the reflexive smooth normed spaces leads to the next result.

\begin{cor}\label{sm}
Let $X$ and $Y$ be real normed spaces with $X$ being reflexive and smooth, and let $f: X \rightarrow Y$ be a phase-isometry. Then for every $x^*\in X^*$, then there exists $\varphi \in Y^*$ with $\|\varphi\|=\|x^*\|$ such that $ x^*(x)=\pm \varphi(f(x))$ for all $x\in X$.
\end{cor}

Now, we are in the position to provide the verification of our first main result.  This solves the Problem \ref{wmu} in positive way when the domain $X$ is a smooth normed space.

\begin{thm}\label{th1}
Let $X$ and $Y$ be real normed spaces with $X$ being smooth, and let $f: X \rightarrow Y$ be a surjective phase-isometry. Then $f$ is phase-equivalent to a linear isometry.
\end{thm}

\begin{prf}
We first prove that $f(z)\in [f(x), f(y)]$ implies that $z\in [x, y]$ for all $x, y, z\in X$. Suppose that, on the contrary, there is $z\notin [x, y]$ satisfying $f(z)\in [f(x), f(y)]$. Set $E:=[x, y, z]$, $F:=[x, y]$ and $d:=d(z, F)=\inf\{\|z-v\|: v\in F\}$. By Hahn-Banach Theorem there is a functional $z_0^* \in S_{E^*}$ such that $z_0^*(z)=d$ and $z_0^*(v)=0$  for all $v\in F$. Note that the mapping $f: E \rightarrow Y$ (the restriction of $f$ to $E$) is a phase-isometry. By the smoothness of $E$ and Corollary \ref{sm}, there exists a linear functional $\varphi\in S_{Y^*}$ such that $z_0^*(u)=\pm \varphi (f(u))$  for all $u\in E$. Write $f(z)=\alpha f(x) + \beta f(y)$ for some $\alpha, \beta \in \mathbb{R}$. Consequently,
\begin{eqnarray*}|z_0^*(z)|=|\varphi(f(z))|=|\alpha \varphi (f(x)) + \beta \varphi(f(y))|=|\alpha z_0^*(x) + \beta z_0^*(y)|=0,\end{eqnarray*} which contradicts the fact that $z_0^*(z)=d>0$.

Now we prove that $f$ satisfies the condition (a) of Proposition \ref{isometry}. For every $x\in X$ and $t\in \mathbb{R}$, we can find a vector $z\in X$ such that $f(z)=tf(x)$. By the prevous result and the fact that $f$ has the norm preserving property, we get $z=\pm tx$. To prove the condition (b), let $x,y\in X$ be linearly independent and set $F:=[x, y]$. Then we can choose two vectors $x_1, x_2\in F$ such that $f(x_1)=f(x)+f(y)$ and  $f(x_2)=f(x)-f(y)$. To see our conclusion, we only need to prove that $x+y\in \{\pm x_1, \pm x_2\}$.
Since $F$ is smooth, Corollary \ref{sm} indicates that for every $x^*\in F^*$ there is a linear functionals $\varphi\in S_{Y^*}$ such that \begin{eqnarray*}x^*(v)=\pm \varphi (f(v))\quad \mbox{for all} \ v\in F.\end{eqnarray*}
In particular, we have
\begin{eqnarray*}
x^*(x+y)&\in&\{\pm \varphi (f(x)+f(y)),  \pm \varphi (f(x)-f(y))\}\\
&=&\{\pm \varphi (f(x_1)),  \pm \varphi (f(x_2))\}\\&=&\{\pm x^*(x_1), \pm x^*(x_2)\}.
\end{eqnarray*}
By the Hahn-Banach separation Theorem, we get $x+y$ belongs to the closed convex hull of $A:=\{\pm x_1, \pm x_2\}$. In fact, the convex hull of $A$
is closed, so we have $x+y\in co(A)$. Note that the linear independence of $x$ and $y$ in $X$ implies that $f(x)$ and $f(y)$ are linearly independent, and so are $x_1$ and $x_2$. Therefore, we can choose two linear functionals ${x_1}^*, {x_2}^*  \in F^*$ such
that ${x_1}^*(x_1)={x_2}^*(x_2)=1$ and ${x_1}^*(x_2)={x_2}^*(x_1)=0$. It follows that $\{{x_1}^*(x+y), {x_2}^*(x+y)\}\subset\{-1, 0, 1\}$. Since $x+y$ belongs to the convex hull of $\{\pm x_1, \pm x_2\}$, we can write $x+y=sx_1+tx_2$ for some $s,t\in\mathbb{R}$ with $|s|+|t|\leq 1$. A short computation shows that $\{s, t\} \subset\{-1, 0, 1\}$, and thus $x+y\in \{\pm x_1, \pm x_2\}$.
\end{prf}

The following result is an immediate consequence of Theorem \ref{th1}.

\begin{cor}
Let $X$ be an $\mathcal{L}^p$-space for $1<p<\infty$, and let $Y$ be a normed space. Suppose that $f: X \rightarrow Y$ is a surjective phase-isometry. Then $f$ is phase-equivalent to a linear isometry.
\end{cor}

\section{Phase-isometries on $\mathcal{L}^\infty(\Gamma)$-type and $\ell^1$-type spaces}

In this section we consider phase-isometries from $\mathcal{L}^\infty(\Gamma)$-type or $\ell^1$-type spaces onto general normed spaces. We shall show that all such mappings are phase equivalent to real linear isometries.

To show the following results of this section, we recall some notations and results about Birkhoff orthogonality. Let $X$ be a real normed space. For all $x,y\in X$, let us denote by $x \bot y$ the \emph{Birkhoff orthogonality} relation on $X$ as:
\begin{eqnarray*}
\|x+t y\|\geq \|x\| \quad \mbox{for all} \ t\in \mathbb{R}.
\end{eqnarray*}
This relation is clearly homogeneous, but neither symmetric nor additive, unless the norm comes from an inner product. Let $A$ be a subset of $X$.
We say that $x \bot A$ ($A\bot x$, respectively) if $x \bot z$ ($z \bot x$, resp) holds for every $z\in A$. The Birkhoff orthogonality can be easily characterized by using the directional derivatives of the norm. For more properties of Birkhoff orthogonality, we refer the reader to the survey paper \cite{AMW} and references therein.

The following is an important property of Birkhoff orthogonality from \cite[Corollary 4.2.]{AMW} and \cite[Theorem 2.1]{J}.

\begin{lem}If $x$ and $y$ are two elements of a normed linear space, then $x \bot y$ if and only if there exists a supporting functional $x^*\in D(x)$ at $x$ such that $x^*(y)=0$. Moreover, for every $x^*\in D(x)$, the set $Z:=\{x\in X: x^*(x)=0\}$ is a hyperplane through the origin such that $x\bot Z$.
\end{lem}

The statement of Proposition \ref{isometry} remains valid if we replace condition (b) by some properties of Birkhoff orthogonality. We will state another sufficient condition to solve the Problem \ref{wmu} for some special normed spaces, which will be of use later.

\begin{prop}\label{bot}
Let $X$ and $Y$ be real normed spaces, and let $f: X \rightarrow Y$ be a surjective phase-isometry satisfying the following conditions:\\
\emph{(a)} $f(t x)=\pm tf(x)$ for every $x\in S_X$ and $t>0$;\\
\emph{(b)} There exist a nonzero vector $x_0\in S_X$ and a hyperplane $Z\subset X$ through the origin such that $x_0\bot Z$ and $Z\bot x_0$. Moreover,  for every $z\in Z$, there exist two real numbers $\alpha(z)$ and $\beta(z)$ with $|\alpha(z)|=|\beta(z)|=1$
such that $f(z+x_0)=\alpha(z) f(z)+\beta(z) f(x_0)$.\\
Then $f$ is phase-equivalent to a linear isometry.
\end{prop}

\begin{prf}
The condition (a) implies that $f$ is phase equivalent to a homogeneous mapping by Remark \ref{weak} and the proof of Proposition \ref{isometry}. Without loss of generality we can assume that $f$ is homogeneous. From condition (b), we write \begin{eqnarray*}f(z+ax_0)=\alpha(\frac{z}{a}) f(z)+ \beta(\frac{z}{a})f(ax_0) ,\ |\alpha(\frac{z}{a})|=|\beta(\frac{z}{a})|=1 \quad (0\neq a\in \mathbb{R}, \ z\in Z ).\end{eqnarray*}
Define a mapping $g: X \rightarrow Y$ as follows:
\begin{eqnarray*}
g(z)=\alpha(z) \beta(z) f(z), \quad g(z+ ax_0)=\alpha(\frac{z}{a}) \beta(\frac{z}{a}) f(z)+ f(ax_0) \quad (0\neq a\in \mathbb{R},\ z\in Z).\end{eqnarray*}
Then $g$ is a phase-isometry, which is phase equivalent to $f$.
Hence, for every $z\in S_Z$ and $a\in\mathbb{R}$, we get
\begin{eqnarray}
\label{b1} &&\{\|g(az+x_0)\pm g(az-x_0)\|\}=\{|2a|, 2\}, \\
\label{b2}&&\{\|g(az+x_0)\pm g(z+x_0)\|\}=\{\|(a+1)z+2x_0\|, |a-1|\},\\
\label{b3}&&\{\|g(az+ax_0)\pm g(az+x_0)\|\}=\{\|2az+(a+1)x_0\|, |a-1|\}.
\end{eqnarray}
Iterating the technique in the proof of Proposition \ref{isometry}, we can prove that
\begin{eqnarray}\label{b4}\alpha(az)\beta(az)=\alpha(z)\beta(z) \quad (0\neq a\in \mathbb{R}, z\in S_Z).\end{eqnarray}
In fact, the equation (\ref{b1}) implies that
\begin{eqnarray*}\alpha(az)\beta(az)=\alpha(-az)\beta(-az), \quad \mbox{for all}\ z\in S_Z, \ 0\neq a\in\mathbb{R}.\end{eqnarray*}
The equation (\ref{b2}) ((\ref{b3}), respectively) implies that $\alpha(az)\beta(az)=\alpha(z)\beta(z)$ for all $z\in S_Z$ and $0<a<1$ ($a>1$, resp). The definition of $g$ and equation (\ref{b4}) show that the mapping $g$ is homogeneous and surjective, and moreover
\begin{eqnarray*}g(z+ax_0)=g(z)+ ag(x_0)\quad \mbox{ for every} \ z\in Z\, \mbox{and} \ a\in\mathbb{R}.\end{eqnarray*} Next, we will show that $g$ is an isometry, and thus $g$ is linear by
the Mazur-Ulam Theorem.
For any two vectors $x_1, x_2\in X$, we write $x_1=z_1+t_1 x_0$ and $x_2=z_2+t_1 x_0$ for some $z_1, z_2\in H$ and $t_1, t_2\in \mathbb{R}$.
Choose a positive number $a>\|x_1\|+\|x_2\|$, we have
\begin{eqnarray*}
&&\{\|g(x_1)+g(x_2)+2ag(x_0)\|,\| g(x_1)-g(x_2)\|\}\\
&=&\{\|g(x_1+ax_0)+g(x_2+ax_0)\|,\| g(x_1+ax_0)-g(x_2+ax_0)\|\}\\
&=&\{\|x_1+x_2+2ax_0\|, \|x_1-x_2\|\}
\end{eqnarray*}
The inequality $\|g(x_1)+g(x_2)+2ag(x_0)\|>2a-\|x_1\|-\|x_2\|>\|x_1-x_2\|$ implies that $\| g(x_1)-g(x_2)\|=\|x_1-x_2\|$ for every
$x_1, x_2\in X$. This completes the proof.
\end{prf}

Let $\Gamma$ be a nonempty index set. The $\ell^{\infty}(\Gamma)$ space is \begin{align*}\{x=\{x_\gamma\}_{\gamma\in\Gamma}:  \| x\| =\sup\limits_{\gamma\in\Gamma} |x_\gamma|<\infty,\ x_\gamma\in\mathbb{R},\ \gamma\in\Gamma \}.
\end{align*}
The $\mathcal{L}^\infty(\Gamma)$-type spaces are the subspaces of $\ell^\infty(\Gamma)$ containing all $e_\gamma {'s}$ ($\gamma\in\Gamma$). For example, the spaces $c_0(\Gamma),c(\Gamma)$ and $\ell^\infty(\Gamma)$ are all $\mathcal{L}^\infty(\Gamma)$-type spaces. For every $x=\{x_\gamma\}_{\gamma\in\Gamma}\in \mathcal{L}^\infty(\Gamma)$, we write $x=\{x_\gamma\}$, and omit the subscripts $\gamma\in\Gamma$ for simplicity of notation. Moreover, we denote the support set of $x$ by $\Gamma_x$, i.e., $$\Gamma_x=\{\gamma\in \Gamma: x_{\gamma} \neq 0 \}.$$
For every $\gamma \in \Gamma$, the coordinate functional $e_{\gamma}^*$ on $\mathcal{L}^\infty(\Gamma)$ defined by $e_{\gamma}^*(x)=x_\gamma$ is the unique supporting functional at $e_{\gamma}$, since $e_{\gamma}$ is a smooth point in $\mathcal{L}^\infty(\Gamma)$.

Let us consider a sign mapping
$\theta: \mathcal{L}^\infty(\Gamma) \rightarrow \ell^\infty(\Gamma)$, given by $\theta (x)=\{\mbox{sign}(x_\gamma)\}$ for every $x=\{x_\gamma\}\in \mathcal{L}^\infty(\Gamma)$. Obviously, $\theta (0)=0$, $\theta(e_{\gamma})=e_{\gamma}$ for every $\gamma\in \Gamma$ and $\theta (x)$ is a norm-one element in $\ell^\infty(\Gamma)$ for every nonzero $x\in \mathcal{L}^\infty(\Gamma)$. Moreover, for all $x, y\in\mathcal{L}^\infty(\Gamma)$ with $\Gamma_x\cap\Gamma_y=\emptyset$ and $t>0$, we have $\theta(x+y)=\theta(x)+\theta(y)$, $\theta(tx)=\theta(x)$ and $\theta(-x)=-\theta(x)$.

\begin{thm}
Let $X$ be an $\mathcal{L}^\infty(\Gamma)$-type space, and let $Y$ be a normed space. Suppose that $f:X\rightarrow Y$ is a surjective phase-isometry. Then $f$ is phase-equivalent to a linear isometry.
\end{thm}

\begin{proof}
The proof will be divided into three steps. The steps (2) and (3) were proved in \cite[Lemmas 2.4 and 2.5]{JT}. A brief sketch of the proof is given below for the reader's convenience.

{\bf Step 1}: By Lemma \ref{diff}, for every $\gamma\in\Gamma$ there exists a linear functional $\varphi_{\gamma} \in S_{Y^*}$ such that
$e_{\gamma}^*(x)=\pm \varphi_{\gamma}(f(x))$ for all $x\in X$.
Moreover, the functionals $\{\varphi_{\gamma}: \gamma\in \Gamma\}$ can be easily modified to obey the condition:
\begin{eqnarray*}
\varphi_{\gamma}(f(e_{\gamma}))=1, \ \varphi_{\gamma}(f(e_{\gamma'}))=0 \quad \mbox{for all}\ \gamma, \gamma'\in \Gamma, \ \gamma'\neq \gamma .
\end{eqnarray*}
We may consider the space $Y$ to be a function space on $\Gamma$ given by $w(\gamma)=\varphi_{\gamma}(w)$ for every $w\in Y,  \gamma\in \Gamma$. For every $x=\{x_\gamma\}\in X$, we have
\begin{eqnarray*}
|f(x)(\gamma)|=|\varphi_{\gamma}(f(x))|=|e_{\gamma}^*(x)|= |x_\gamma|.
\end{eqnarray*}
Consequently,
\begin{eqnarray*}
\sup\{|f(x)(\gamma)|: \gamma\in \Gamma\}=\sup\{|x_{\gamma}|: \gamma\in \Gamma\}=\|x\|=\|f(x)\|.
\end{eqnarray*}
Thus we can regard $Y$ as a subspace in $\ell^\infty(\Gamma)$. Moreover, for every $x=\{x_\gamma\}\in X$, we obtain the formula $f(x)=\{f(x)_{\gamma}\}\in Y$, where $|f(x)_{\gamma}|=|x_{\gamma}|$ for every $\gamma\in \Gamma$.

{\bf Step 2}: We will prove that $\theta(f(x))=\pm f(\theta(x))$  for every $x\in X$.
From the equation (\ref{phase}),
\begin{eqnarray*}
&&\{\|f(\theta(x))\pm f(\|x\| \theta(x))\|\}=\{\|\theta(x)\pm \|x\| \theta(x)\|\}=\{1+\|x\|, |1-\|x\||\};\\
&&\{\|f(\|x\|\theta(x))\pm f(x)\|\}=\{\|\|x\|\theta(x) \pm x\|\}=\{2\|x\|, \|x\|-\inf_{\gamma\in \Gamma_x}|x(\gamma)|\}.
\end{eqnarray*}
It follows that
\begin{eqnarray*}
 f(\theta(x))=\pm \theta(f(\|x\| \theta(x))) \ \mbox{and} \ \theta(f(\|x\|\theta(x)))=\pm \theta(f(x)).
\end{eqnarray*}

{\bf Step 3}: We will prove that the mapping $f$ satisfies conditions (a) and (b) of Proposition \ref{bot} and this completes the proof.
To check the condition (a), it suffices to show that $\theta(f(tx))=\pm \theta(f(x))$ for every $x\in S_X$ and $t>0$.
We apply the result in the Step 2 to obtain
\begin{eqnarray*}
\theta(f(tx))=\pm f(\theta(tx))=\pm f(\theta(x))=\pm \theta(f(x)).
\end{eqnarray*}
It remains to check the condition (b). For fixed $\gamma_0\in \Gamma$, set $Z:=\{x\in X: e_{\gamma_0}^*(x)=0\}$. Then $Z$ is a hyperplane through the origin such that $e_{\gamma_0}\bot Z$ and $Z\bot e_{\gamma_0}$. We need to check that for every $z\in Z$, there exist two real numbers $\alpha(z)$ and $\beta(z)$ with $|\alpha(z)|=|\beta(z)|=1$ such that $f(z+ e_{\gamma_0})=\alpha(z) f(z) + \beta(z) f(e_{\gamma_0})$. This is equivalent to showing that there exist two real numbers $\alpha$ and $\beta$ with $|\alpha|=|\beta|=1$ such that \begin{eqnarray*}\theta(f(z+ e_{\gamma_0}))=\alpha \theta(f(z)) + \beta f(e_{\gamma_0}).\end{eqnarray*} Note that $\theta(f(z+ e_{\gamma_0}))=\pm f(\theta(z+e_{\gamma_0}))=\pm f(\theta(z)+e_{\gamma_0}).$
From the step 1, we write \begin{eqnarray*} f(\theta(z)+e_{\gamma_0})=\{b_{\gamma} \}+ \beta f(e_{\gamma_0}) \ \mbox{and} \ f(\theta(z))=\{c_{\gamma}\},\end{eqnarray*} where $|\beta|=1$ and $|b_{\gamma}|=|c_{\gamma}|=1$ for every $\gamma\in \Gamma_z$. For every $a,b\in\mathbb{R}$, we put $a\vee b=\max\{a,b\}$.
Then \begin{eqnarray*}&&\{\sup_{\gamma\in \Gamma_z}\{|b_{\gamma}+c_{\gamma}|\}\vee 1, \sup_{\gamma\in \Gamma_z}\{|b_{\gamma}-c_{\gamma}|\}\vee 1\}\\&=&\{\|f(\theta(z)+e_{\gamma_0})+f(\theta(z))\|, \|f(\theta(z)+e_{\gamma_0})-f(\theta(z))\|\}\\&=&\{\|2\theta(z)+e_{\gamma_0}\|, \|e_{\gamma_0}\|\}=\{2, 1\} .\end{eqnarray*}
It follows that $f(\theta(z)+e_{\gamma_0})=\pm f(\theta(z)) +\beta f(e_{\gamma_0})$, as desired.
\end{proof}

Let $\Gamma$ be a nonempty index set. The $\ell^1(\Gamma)$ space is
\begin{eqnarray*}&& \{x=\sum_{\gamma\in \Gamma} x_{\gamma} e_{\gamma}: \|x\|=\sum_{\gamma\in \Gamma} |x_{\gamma}|<\infty, \ x_\gamma \in \mathbb{R}, \ \gamma\in \Gamma\}.\end{eqnarray*}
For every $x\in \ell^1(\Gamma)$, we denote the support set of $x$ by $\Gamma_x$, i.e.,
$$\Gamma_x=\{\gamma\in \Gamma: x_{\gamma} \neq 0 \}.$$  For all $x,y \in \ell^1(\Gamma)$, we have $x\bot y$ if and only if $\Gamma_x \cap \Gamma_y=\emptyset$. It is well known that for all $x,y \in \ell^1(\Gamma)$,
\begin{eqnarray*}x \bot y \Leftrightarrow \|x+y\|+\|x-y\|=2(\|x\|+\|y\|)\Leftrightarrow \|x+y\|=\|x-y\|=\|x\|+\|y\|. \end{eqnarray*}
Please note that for every index set $\Gamma$, if $X=\ell^1(\Gamma)$ then $X^*=\ell^{\infty}(\Gamma)$. Moreover, if $\Gamma$ is a countable set then all $w^*$-exposed points of $B_{\ell^{\infty}(\Gamma)}$ are just \begin{eqnarray*}\{x=\{\theta_{\gamma} e_{\gamma}\}: \theta_{\gamma}=\pm 1, \gamma\in \Gamma\}.\end{eqnarray*} For every uncountable set $\Gamma$, there are no $w^*$-exposed points of $B_{\ell^{\infty}(\Gamma)}$.

The next main result needs a lemma whose proof depends on  the ``continuity'' of phase-isometries.
\begin{lem}\label{Banach}
Let $X$ and $Y$ be real normed spaces, and let $f: X \rightarrow Y$ be a surjective phase-isometry. If $X$ is a Banach space, then so
 is $Y$.
\end{lem}

\begin{prf}
Suppose that $\{f(x_n): x_n\in X, n\in \mathbb{N}\}$  is a Cauchy sequence in $Y$. Clearly, $\{\|x_n\|\}$ is also a Cauchy sequence
in $\mathbb{R}$. If $\{\|x_n\|\}$ converges to zero, then $\lim_{n\rightarrow\infty}f(x_n)=0$ in $Y$ by norm preserving property.
Now we assume that $a=\lim_{n\rightarrow\infty}\|x_n\|>0$. For every $0<\epsilon<a/2$, there is a cutoff $n_{\epsilon}\in \mathbb{N}$ such that whenever $m,n>n_{\epsilon}$, we have $\|x_m\|, \|x_n\|>a-\epsilon/2$ and $\|f(x_n)-f(x_m)\|< \epsilon$.
From the equation (\ref{phase}),  \begin{eqnarray*}\|f(x_n)+f(x_m)\|+\|f(x_n)-f(x_m)\|=\|x_n+x_m\|+\|x_n-x_m\|\geq 2\max\{\|x_n\|, \|x_m\|\}\end{eqnarray*} for all $m,n\in \mathbb{N}$. It follows that $\|f(x_n)+f(x_m)\|>2a-2\epsilon>a$ for all $m,n>n_{\epsilon}$.
Fix a positive integer $n_0>n_{\epsilon}$ and set $A:=\{m \in \mathbb{N}: \|x_m-x_{n_0}\|=\|f(x_m)-f(x_{n_0})\|<\epsilon, m>n_0\}$ and $B:=\{m \in \mathbb{N}: \|x_m+x_{n_0}\|=\|f(x_m)-f(x_{n_0})\|<\epsilon, m>n_0\}$. Then for all $m, n\in A$ or $B$, we have \begin{eqnarray*}\|x_m-x_n\|\leq\|f(x_{m})-f(x_{n_0})\|+\|f(x_{n})-f(x_{n_0})\|< 2\epsilon<a.\end{eqnarray*} The inequality $\|f(x_n)+f(x_m)\|>a$  yields $\|x_m-x_n\|=\|f(x_m)-f(x_n)\|$ for all $m, n\in A$ or $B$. Since $A$ or $B$ is an infinite set, we may assume by passing to a subsequence that $\{x_n\}$ is a Cauchy sequence, and $\{x_n\}$ converges to $x$ in $X$. From the equation (\ref{phase}), we have \begin{eqnarray*}\{\|f(x_n)-f(x)\|, \|f(x_n)+f(x)\|\}=\{\|x_n-x\|,\|x_n+x\|\}\end{eqnarray*} for every $n\in \mathbb{N}$. So we obtain that the Cauchy sequence $\{f(x_n)\}$ converges to $f(x)$ or $-f(x)$. This completes the proof.
\end{prf}

\begin{rem}
Let $X$ and $Y$ be real normed spaces, and let $f: X \rightarrow Y$ be a phase-isometry. Using the same proof as the above Lemma \ref{Banach}, for every sequence  $\{x_n\}$ that converges to $x$ in $X$, there is a subsequence $\{f(x_{n_i})\}$ converges to $f(x)$ or $-f(x)$ in $Y$.
\end{rem}

\begin{thm}
Let $X=\ell^1(\Gamma)$, and let $Y$ be a normed space. Suppose that $f:X\rightarrow Y$ is a surjective phase-isometry. Then $f$ is phase-equivalent to a linear isometry.
\end{thm}

\begin{prf}
The proof will be divided into five steps.

{\bf Step 1}: Since $f$ satisfies the equation (\ref{phase}), it follows that for all $x,y\in X$,
\begin{eqnarray*}x\bot y \Leftrightarrow\|f(x)+f(y)\|=\|f(x)-f(y)\|=\|f(x)\|+\|f(y)\|.\end{eqnarray*}
Alternatively, $x\bot y$ if and only if $\|af(x)+bf(y)\|=|a|\|f(x)\|+|b|\|f(y)\|$ for all $a,b\in \mathbb{R}$. Fix $\gamma_0\in \Gamma$ and $t\in \mathbb{R}$, we can find $x\in X$ with $\|x\|=|t|$ such that $f(x)=tf(e_{\gamma_0})$.  Then for every $\gamma\in \Gamma$ with $\gamma\neq \gamma_0$,
\begin{eqnarray*}\|x+e_{\gamma}\|+\|x-e_{\gamma}\|=\|tf(e_{\gamma_0})+f(e_{\gamma})\|+\|tf(e_{\gamma_0})-f(e_{\gamma})\|=2(1+\|x\|). \end{eqnarray*}
This yields $x\bot e_{\gamma}$ for every $\gamma\in \Gamma$ and $\gamma\neq \gamma_0$, and hence $x=\pm te_{\gamma_0}$. We
conclude that \begin{eqnarray*} f(te_\gamma)=\pm tf(e_\gamma), \quad \mbox{for all} \ \gamma\in \Gamma, \ t\in \mathbb{R}.\end{eqnarray*}

{\bf Step 2}: We shall show by induction that for every $n\in \mathbb{N}$, $\{\gamma_i\in \Gamma: 1\leq i\leq n\}$ and $\{b_i\}_{i=1}^{n}\subset \mathbb{R}$, if $f(x)=\sum_{i=1}^{n} b_i f(e_{\gamma_i})\in Y$ for some $x\in X$ then \begin{eqnarray*}x=\sum_{i=1}^{n} a_i e_{\gamma_i}, \quad \mbox{where} \ |a_i|=|b_i|\ \mbox{for all} \ 1\leq i\leq n.\end{eqnarray*} The statement is true for $n=1$ from Step 1. Suppose that the conclusion is true for $n=k-1$. Now take $x=\sum_{\gamma\in \Gamma} x_{\gamma} e_{\gamma} \in X$ such that $f(x)=\sum_{i=1}^{k} b_i f(e_{\gamma_i})$. For every $1\leq m\leq k$, we can choose $y_m\in X$ such that $f(y_m)=f(x)-b_mf(e_{\gamma_m})$. By the assumption of the truth of $n=k-1$, we have $y_m\bot e_{\gamma_m}$ and $\|y_m\|=\sum_{i=1}^k|b_i|-|b_m|$. It follows that $\|f(x)\|=\sum_{i=1}^k|b_i|$, and hence \begin{eqnarray*} &&\|f(x)+b_mf(e_{\gamma_m})\|+\|f(x)-b_mf(e_{\gamma_m})\|\\ &=& \|f(y_m)+2b_mf(e_{\gamma_m})\|+\|f(y_m)\|\\&=&\|y_m\|+|2b_m|+\|y_m\|\\ &=&2\|f(x)\|=2\|x\|.\end{eqnarray*}
On the other hand,
\begin{eqnarray*}&&\|f(x)+b_mf(e_{\gamma_m})\|+\|f(x)-b_mf(e_{\gamma_m})\|\\&=&\|x+b_me_{\gamma_m}\|+\|x-b_me_{\gamma_m}\|\\ &=&2(\|x\|-|x_{\gamma_m}|)+|x_{\gamma_m}+b_m|+|x_{\gamma_m}-b_m|\\ &=& 2\|x\|-2\|x_{\gamma_m}\|+2\max\{\|x_{\gamma_m}\|,|b_m|\|\}.
\end{eqnarray*}
It follows that $|x_{\gamma_m}|\geq |b_m|$ for every $1\leq m\leq k$. Observe that \begin{eqnarray*}\sum_{\gamma\in \Gamma} |x_{\gamma}|=\|x\|=\|f(x)\| =\sum_{i=1}^k|b_i|.\end{eqnarray*} Therefore, $|x_{\gamma_m}|=|b_m|$ for all $1\leq m\leq k$ and $x_{\gamma}=0$ for every $\gamma\in \Gamma\setminus\{\gamma_n: 1\leq n\leq k\}$. Consequently, by the Principle of Finite Induction, the statement is true for all $n\in \mathbb{N}$.

{\bf Step 3}: We shall prove that for all sequences $\{\gamma_n\}_{n=1}^{\infty}\subset\Gamma$ and $\{b_n\}_{n=1}^{\infty}\subset \mathbb{R}$ with $\sum_{n=1}^{\infty}|b_n|<\infty$, if $f(x)=\sum_{n=1}^{\infty} b_n f(e_{\gamma_n})\in Y$ for some $x\in X$ then \begin{eqnarray*}x=\sum_{n=1}^{\infty} a_n e_{\gamma_n}, \quad \mbox{where} \ |a_n|=|b_n|\ \mbox{for all} \ n\in \mathbb{N}.\end{eqnarray*}
Since $Y$ is a Banach space by Lemma \ref{Banach}, the formula $\sum_{n=1}^{\infty} b_n f(e_{\gamma_n})$ is well defined. Step 2 shows that for every sequence $\{c_n\}_{n=1}^{\infty}\subset \mathbb{R}$ with $\sum_{n=1}^{\infty}|c_n|<\infty$,
\begin{eqnarray*}\|\sum_{n=1}^{\infty} c_n f(e_{\gamma_n})\|=\lim_{k\rightarrow \infty}\|\sum_{n=1}^{k} c_n f(e_{\gamma_n})\|=\lim_{k\rightarrow \infty}\sum_{n=1}^{k}|c_n|=\sum_{n=1}^{\infty}|c_n|.\end{eqnarray*}
We write $x=\sum_{\gamma\in \Gamma} x_{\gamma} e_{\gamma} \in X$. A short computation shows that  \begin{eqnarray*} 2\|f(x)\|&=&\|f(x)+b_nf(e_{\gamma_n})\|+\|f(x)-b_nf(e_{\gamma_n})\|\\&=&\|x+b_ne_{\gamma_n}\|+\|x-b_ne_{\gamma_n}\|\\ &=&2\|x\|-2\|x_{\gamma_n}\|+2\max\{\|x_{\gamma_n}\|,|b_n|\|\} \end{eqnarray*} for every $n\in \mathbb{N}$.
This means that $|x_{\gamma_n}|\geq |b_n|$ for every $n\in \mathbb{N}$. Moreover, the equation $\|x\|=\sum_{n=1}^{\infty}|b_n|$  proves that
$x=\sum_{n=1}^{\infty} x_{\gamma_n} e_{\gamma_n}$, where $|x_{\gamma_n}|=|b_n|$ for every $n\in \mathbb{N}$.

{\bf Step 4}: Now we shall prove that if $x=\sum_{\gamma\in \Gamma} x_{\gamma} e_{\gamma} \in X$ then $f(x)=\sum_{\gamma\in \Gamma} f(x)_{\gamma} f(e_{\gamma})$, where $|f(x)_{\gamma}|=|x_{\gamma}|$ for every $\gamma\in \Gamma$. For every $x\in X$, the support set $\Gamma_x$ is countable. Since the proof also applies to the case of finite, we can put $\Gamma_x=\{\gamma_n: n\in \mathbb{N}\}$ and $x=\sum_{n=1}^{\infty}x_{\gamma_n} e_{\gamma_n}$. Let $E:=\ell^1(\Gamma_x)$ be a subspace of $X$ and $x^*=\sum_{n=1}^{\infty}\mbox{sign}(x_{\gamma_n}) e_{\gamma_n}$ be a linear functional in $E^*=\ell^{\infty}(\Gamma_x)$. Then $x^*$ is a $w^*$-exposed points of $B_{E^*}$ and \begin{eqnarray*}x^*(x)=\|x\|=\sum_{n=1}^{\infty}|x_{\gamma_n}|.\end{eqnarray*}
Note that the restriction of $f$ to $E$ is also a phase-isometry. By Lemma \ref{diff}, there exists a linear functional $\varphi \in S_{Y^*}$ such that $x^*(v)=\pm \varphi(f(v))$ for all $v\in E$. In particular, $\varphi(f(e_{\gamma_n}))=\pm x^*(e_{\gamma_n})=\pm 1$ for every $n\in \mathbb{N}$. Since $Y$ is a Banach space, we get the vector $w=\sum_{n=1}^{\infty} \varphi(f(e_{\gamma_n})) |x_{\gamma_n}| f(e_{\gamma_n})\in Y$ is well defined. By Step 3, we can find $z=\sum_{n=1}^{\infty}z_{\gamma_n} e_{\gamma_n}\in E$ such that $f(z)=w$, where $|z_{\gamma_n}|=|\varphi(f(e_{\gamma_n})) x_{\gamma_n}|=|x_{\gamma_n}|$ for every $n\in \mathbb{N}$. The equation $|x^*(z)|=|\varphi(f(z))|=\sum_{n=1}^{\infty}|x_{\gamma_n}|$ implies that $z=\pm x$, and so
\begin{eqnarray*} f(x)=\pm \sum_{n=1}^{\infty} \varphi(f(e_{\gamma_n})) |x_{\gamma_n}| f(e_{\gamma_n}).\end{eqnarray*}

{\bf Step 5}: We check that the mapping $f$ satisfies conditions (a) and (b) of Proposition \ref{bot}. To check the condition (a).
Let $x=\sum_{\gamma\in \Gamma_x} x_{\gamma} e_{\gamma}\in S_X$ and $t>0$. By Step 4 we can write \begin{eqnarray*}f(x)=\sum_{\gamma\in \Gamma_x} b_{\gamma} f(e_\gamma) \ \mbox{and} \ f(tx)=t\sum_{\gamma\in \Gamma_x} c_{\gamma} f(e_\gamma),\end{eqnarray*} where $|b_{\gamma}|=|c_{\gamma}|=|x_{\gamma}|$ for every $\gamma\in \Gamma_x$. Therefore,
\begin{eqnarray*}1+t=\|x+tx\| &\in& \{\|f(x)+f(tx)\|, \|f(x)-f(tx)\|\}\\ &=&\{\sum_{\gamma\in \Gamma_x}|b_{\gamma}+tc_{\gamma}|, \sum_{\gamma\in \Gamma_x}|b_{\gamma}-tc_{\gamma}|\}.\end{eqnarray*}
From the inequality \begin{eqnarray*}\sum_{\gamma\in \Gamma_x}|b_{\gamma}\pm tc_{\gamma}|\leq \sum_{\gamma\in \Gamma_x}|b_{\gamma}|+t|c_{\gamma}|=1+t,\end{eqnarray*} we conclude that $f(tx)=\pm tf(x)$. To check the condition (b). For fixed $\gamma_0\in \Gamma$, the set $Z:=\{z\in \ell^1(\Gamma): e_{\gamma_0}\bot z\}$ is a hyperplane through the origin such that $e_{\gamma_0}\bot Z$ and $Z\bot e_{\gamma_0}$. For every $z= \sum_{\gamma\in \Gamma_z} z_{\gamma} e_{\gamma}$, we write \begin{eqnarray*} f(z+e_{\gamma_0})=\sum_{\gamma\in \Gamma_x} b_{\gamma} f(e_\gamma)+\beta f(e_{\gamma_0})\ \mbox{and} \ f(z)=\sum_{\gamma\in \Gamma_z} c_{\gamma} f(e_\gamma) ,\end{eqnarray*} where $|\beta|=1$ and $|b_{\gamma}|=|c_{\gamma}|=|z_{\gamma}|$ for every $\gamma\in \Gamma_z$. Consequently, \begin{eqnarray*}1=\|z+e_{\gamma_0}\|-\|z\|&\in& \{\|f(z+e_{\gamma_0})+f(z)\|, \|f(z+e_{\gamma_0})-f(z)\|\}\\ &=&\{\sum_{\gamma\in \Gamma_x}|b_{\gamma}+c_{\gamma}|+1, \sum_{\gamma\in \Gamma_x}|b_{\gamma}-c_{\gamma}|+1\}.\end{eqnarray*}
This means $f(z+e_{\gamma_0})=\pm f(z)+ \beta f(e_{\gamma_0})$ for every $z\in Z$, which completes the proof.
\end{prf}

\subsection*{Acknowledgements}
The authors wish to express their appreciation to Guanggui Ding for many very helpful comments regarding isometric theory in Banach spaces.

\end{document}